\pdfoutput=1
\documentclass[11pt]{article}
\usepackage{vmargin}
\usepackage{times}

\usepackage[T1]{fontenc}
\usepackage{epsfig}
\usepackage{graphicx}
\usepackage{amssymb}
\usepackage{amsmath}
\usepackage{latexsym} 
\usepackage[utf8]{inputenc}
\usepackage{amsfonts}
\usepackage{amsthm}
\usepackage{paralist}
\usepackage{stmaryrd}

\usepackage[colorlinks=true]{hyperref}

\usepackage{microtype}

\theoremstyle{plain}
\newtheorem{corollary}{Corollary}
\newtheorem{theorem}[corollary]{Theorem}
\newtheorem{lemma}[corollary]{Lemma}
\newtheorem{proposition}[corollary]{Proposition}
\newtheorem*{theorem*}{Theorem}
\newtheorem*{lemma*}{Lemma}
\newtheorem*{definition*}{Definition}
\newtheorem*{corollary*}{Corollary}

\theoremstyle{definition}

\theoremstyle{remark}

\newtheorem*{remark}{Remark}

\newtheorem{algorithm}{Algorithm}

\newcommand{\R}{\mathbb{R}}

\newcommand{\proba}{\mathbb{P}}
\newcommand{\N}{\mathbb{N}}
\newcommand{\E}{\mathbb{E}}

\newcommand{\C}{\mathcal{C}}

\newcommand{\T}{\mathcal{T}}

\renewcommand{\L}{\mathcal{L}}

\DeclareMathOperator*{\diam}{diam}

\newcommand*{\height}{\mathfrak h}

\newcommand{\set}[1]{\llbracket 1,#1\rrbracket}
\begin{document}
\title{Diameter and long paths in critical digraph}
\author{Arthur Blanc-Renaudie\thanks{Tel Aviv University, Israel, ablancrenaudiepro@gmail.com}}
\date{\today}
\maketitle
\begin{abstract} We study the random directed graph $\vec G(n,p)$ in which each of the $n(n-1)$ possible directed edges are present with probability $p$. We show that in the critical window the longest self avoiding oriented paths in $\vec G(n,p)$ have length $O_{\proba}(n^{1/3})$ so $\vec G(n,p)$ has diameter $O_{\proba}(n^{1/3})$.
\end{abstract}
\section{Introduction}
Since its introduction more than fifty years ago \cite{ER} the interest for Erd\"os--R\'enyi random graphs has kept growing. This model and its study has also been generalized in many ways due to many critics. An important one is that many real-life networks are oriented. However, random digraphs, which are the natural oriented version of Erd\'os--R\'enyi random graphs, are much harder to study due to their lack of symmetry, and thus less understood. 

For $n\in \N,p\in[0,1]$, the random digraph $\vec G(n,p)$ is a random directed graph on $\{1,2,\dots,n\}$ in which each of the $n(n-1)$ possible directed edges are independently  present with probability $p$.

In this paper, we are interested in the phase transition for random digraphs. The first result on this phase transition was proved by Karp \cite{KarpOriented} and Luczak \cite{LuczakOriented}. Those results was then refined by Luczak and Seierstad \cite{LuczakSeierstad} who showed that the size of strongly connected components in random digraphs share the same critical window as Erd\"os--R\'enyi random graphs. Later Goldshmidt and Stephenson \cite{ChristinaOriented} proved that renormalized by $n^{1/3}$ the strongly connected components converge. 

Here we consider the whole digraph by studying its diameter and long paths. More precisely we say that $(v_0,v_1,\dots, v_k)$ is a saop (for Self-Avoiding Oriented Path) of length $k$ in an oriented graph $\vec G$  if the oriented edges $(v_0,v_1),\dots (v_{k-1},v_k)$ are different and in $\vec G$. For every $i,j\in \vec G$, let $d_\rightarrow(i,j)$ be the minimal length of a saop from $i$ to $j$ (or by default $-\infty$ if there is not such paths). The diameter of $\vec G$ is $\diam(\vec G):=\max_{i,j\in \vec G} d_\rightarrow(i,j)$.
 
For a sequence $(X_n)_{n\in \N}$ of random variables and $f:n\mapsto \R^{+*}$, we say that $X_n=O_{\proba}(f(n))$ if $(X_n/f(n))_{n\in \N}$ is tight, and that $X_n=o_{\proba}(f(n))$ if $(X_n/f(n))_{n\in \N}$ converges in probability to $0$. 

We will assume in all the paper that we are in the critical window, that is $p=1/n+O(1/n^{4/3})$.
The goal of this paper is to show the next results:  (see also Theorem \ref{thm3} for more precise bounds)
 \begin{theorem} \label{Thm1} As $n\to \infty$, $ \diam(\vec G(n,p))=O_{\proba}(n^{1/3}) $.
\end{theorem}
\begin{theorem} \label{Thm2} As $n\to \infty$, the longest saop in $\vec G(n,p)$ have length $O_{\proba}(n^{1/3})$.
\end{theorem}
One of our main motivations for the above Theorems is to show scaling limits for $\vec G(n,p)$, as proved for Erd\"os--R\'enyi random graphs by Addario--Berry, Broutin, and Goldshmidt in \cite{ABG}. However, one would first need to develop a topological theory for oriented metric spaces, and generalize the Gromov-Hausdorff--Prokhorov topology for those spaces. Solving this problem would thus require a whole new set of ideas. 

\paragraph{Plan of the paper} In Section \ref{sec:Explain}, we prove Theorems \ref{Thm1} and \ref{Thm2} assuming an analog result for uniform rooted tree. This result is then proved in Section  \ref{Sec:Chaining} using the chaining method and branch tightness. 

\paragraph{Acknowledgment} This project started in a 2022 Barbados Workshop, where I had the occasion to meet many nice people and discuss the scaling limit problem with them. Thanks! I would also like to apology  for a previous wrong proof that the longest saop would have length of order $\log(n) n^{1/3}$. I am supported by the ERC consolidator grant 101001124 (UniversalMap).
\section{Reduction to a tree problem} \label{sec:Explain}
\subsection{The long path probability, and why it implies Theorems \ref{Thm1} and \ref{Thm2}} \label{2.1}
We show here that the next technical proposition implies Theorem \ref{Thm1} and \ref{Thm2}. In the rest of the section, we will explain why it comes from an analog result for uniform rooted trees. For every $n,m\in \N$, let $\llbracket n,m\rrbracket:=[n,m]\cap \N$. For  every $v\in \set n$, we write $\L^{\uparrow}(v)$ (resp. $\L^{\downarrow}(v)$) for the length of the longest saop in $\vec G(n,p)$ starting from (resp. ending at) $v$. 
\begin{proposition} \label{Main} For every $\delta>0$, there exists $C\in \R^+$, such that for every $n\in \N$, $v\in \set n$,
\[ \proba(\L^{\uparrow}(v)> \delta n^{1/3})=\proba(\L^{\downarrow}(v)> \delta n^{1/3}) \leq C n^{-1/3}. \]
\end{proposition} 
Given Proposition \ref{Main}, the BK inequality implies Theorem \ref{Thm2}, and so Theorem \ref{Thm1}:
\begin{proof}[Proof of Theorem \ref{Thm2}] We say that $(u,v)$ is an inner edge if disjointly: $(u,v)\in \vec G(n,p)$, there is a saop of length at least $n^{1/3}$ ending at $u$, and there is a saop of length at least $n^{1/3}$ starting from $v$. By Proposition \ref{Main}, and by the BK inequality, for some constant $C$, for every $u,v\in \set n$,  
\[ \proba((u,v)\text{ is an inner edge})\leq pC^2 n^{-2/3}. \]
Thus, by Markov's inequality the number of inner edges is $O_{\proba}(n^{1/3})$. Theorem \ref{Thm2} follows as in a saop of length $Ln^{1/3}$ there is at least $(L-2)n^{1/3}-3$ inner edges. 
\end{proof}
Refining the above proof, the BK inequality also directly provides sub-gamma upper-bound for $\L(\vec G(n,p))$ the length of the longest saop in $\vec G(n,p)$, and so for $\diam( \vec G(n,p))\leq \L(\vec G(n,p))$.
\begin{theorem} \label{thm3} There exists $c,C>0$ such that for every $n\in \N$, for every $x>C$,
\[ \proba(\diam (\vec G(n,p))>x n^{1/3})\leq \proba(\L (\vec G(n,p))>x n^{1/3}) \leq e^{-cx\log x}.\]
\end{theorem}
\begin{remark} I don't think the above bound is optimal, as for Erd\"os--R\'enyi random graphs a bound in $e^{-c x^{3/2}}$ was proved by Nachmias and Peres \cite[Proposition 1.4]{NachmiasPeres}.
\end{remark}
\begin{proof} For $k\in \N$, we say that $(u_1,v_1),(u_2,v_2),\dots, (u_k,v_k)\in \set n$ are disjointly inner edges, if disjointly for $1\leq i \leq k$, $(u_i,v_i)$ is an inner edge. By Proposition \ref{Main}, and by the BK inequality, for some $C\in \R^+$, for every $k\in \N$, $u_1,v_1,u_2,v_2,\dots,u_k,v_k\in \set n$,  
\begin{equation} \proba((u_1,v_1),\dots, (u_k,v_k) \text{ are disjointly inner edges})\leq p^kC^{2k} n^{-2k/3}. \label{30/04/14h} \end{equation}

Then given a saop of length $L$, we may look at the $k$-tuples of edges in this path that are separated by at least $2\lceil n^{1/3}\rceil$ edges from each other and separated by at least $\lceil n^{1/3}\rceil$ edges from the start and end of the saop. By a counting argument, we see that if $L>9k n^{1/3}$ then there is at least $(L/2)^k$ such $k$-tuples of edges. Moreover, note that such tuples of edges are disjointly inner edges. 

It follows from \eqref{30/04/14h} and by Markov's inequality, that for every $k\in \N$, $L>9k n^{1/3}$,
\[ \proba(\L(\vec G(n,p))\geq L)\leq n^{2k}p^k C^{2k} n^{-2k/3}/ (L/2)^k = (2pn^{-2/3}C/L)^k.\]
The desired result follows by taking $L=\lceil xn^{1/3}\rceil $ and $k=\lfloor x/9 \rfloor$.
\end{proof}
\subsection{The depth-first search (DFS) exploration} \label{sec:DFS}  \label{2.2} 
To show Proposition \ref{Main} we need a construction from \cite{ChristinaOriented} for $\vec G(n,p)$. Since we only look at the paths going to $1$, we only need part of it. So, we aim for concision, and refer to \cite{ChristinaOriented} for details. We define the DFS exploration of a graph as follows: We define inductively for $i\in \llbracket 0,n\rrbracket$ an ordered set (or stack) $\mathcal O_i$ of open vertices at time $i$, and a set $\mathcal A_i$ of vertices already explored at time $i$.

\begin{algorithm}\label{DFS} \emph{DFS exploration of a graph on $\set n$.}
\begin{compactitem}
\item Set $\mathcal O_0=(1)$, $\mathcal A_0=\emptyset$.
\item For $0\leq i \leq n-1$: Let $v_{i+1}$ be the first vertex of $\mathcal O_i$. Let $\mathcal A_{i+1}= \mathcal A_i\cup\{v_{i+1}\}$. Let $\mathfrak N_i$ be the set of neighbours of $v_i$ that are not in $\mathcal O_i\cup \mathcal A_i$. Construct $\mathcal O_{i+1}$ by removing $v_i$ from the start of $\mathcal O_i$ and by adding in increasing order $\mathfrak N_i$ to the start of $\mathcal O_{i}\backslash \{v_i\}$. If now $\mathcal O_{i+1}=\emptyset$, add to it the lowest-labelled element of $\set n\backslash \mathcal A_i$. 
\end{compactitem}
\end{algorithm} 
Note that alongside the exploration we construct a permutation $(v_i)_{1\leq i \leq n}$ of $\set n$, which we call the DFS-order. Our main motivation for using this DFS exploration is the next construction:
\begin{algorithm} \label{algo2} Decomposition of $\vec G=\vec G(n,p)$ from Proposition 2.1 (iii)  of \cite{ChristinaOriented}:
\begin{compactitem}
\item Let $G(n,p)$ be a Erd\"os--R\'enyi random graph.
\item Let $(v_i)_{1\leq i \leq n}$ be the DFS order for $G(n,p)$.
\item For every $i<j$ put $(v_j,v_i)$ in $\vec G(n,p)$ if and only if $\{v_j,v_i\}\in G(n,p)$.
\item Independently for every $i<j$, put $(v_i,v_j)$ in $\vec G(n,p)$ with probability $p$.
\end{compactitem}
\end{algorithm}
Along this construction we split the oriented edges of $\vec G(n,p)$ in three categories (see Figure \ref{DFSfig}): 
\begin{compactitem}
\item The down-edges: Let $\downarrow(\vec G):=\{(v_i,v_j),i\in \set n, j\in \mathfrak N_i\}$ be the set of down-edges of $\vec G$. Note that $(\set n,\downarrow(\vec G))$ is a forest oriented from the leaves toward the roots. 
\item The surplus-edges:  Let  $\hookleftarrow(\vec G):=\{(v_j,v_i),i<j\}\backslash \downarrow(\vec G)$ be the set of surplus edges of $\vec G$. One can recover $G(n,p)$ by removing the orientations in $\downarrow(\vec G)\cup \hookleftarrow(\vec G)$.
\item The right-edges: Let  $\leadsto(\vec G):=\{(v_i,v_j),i<j\}$ be the set of right-edges. 
\end{compactitem}
 \begin{figure}[!h] 
\centering
\includegraphics[scale=0.60]{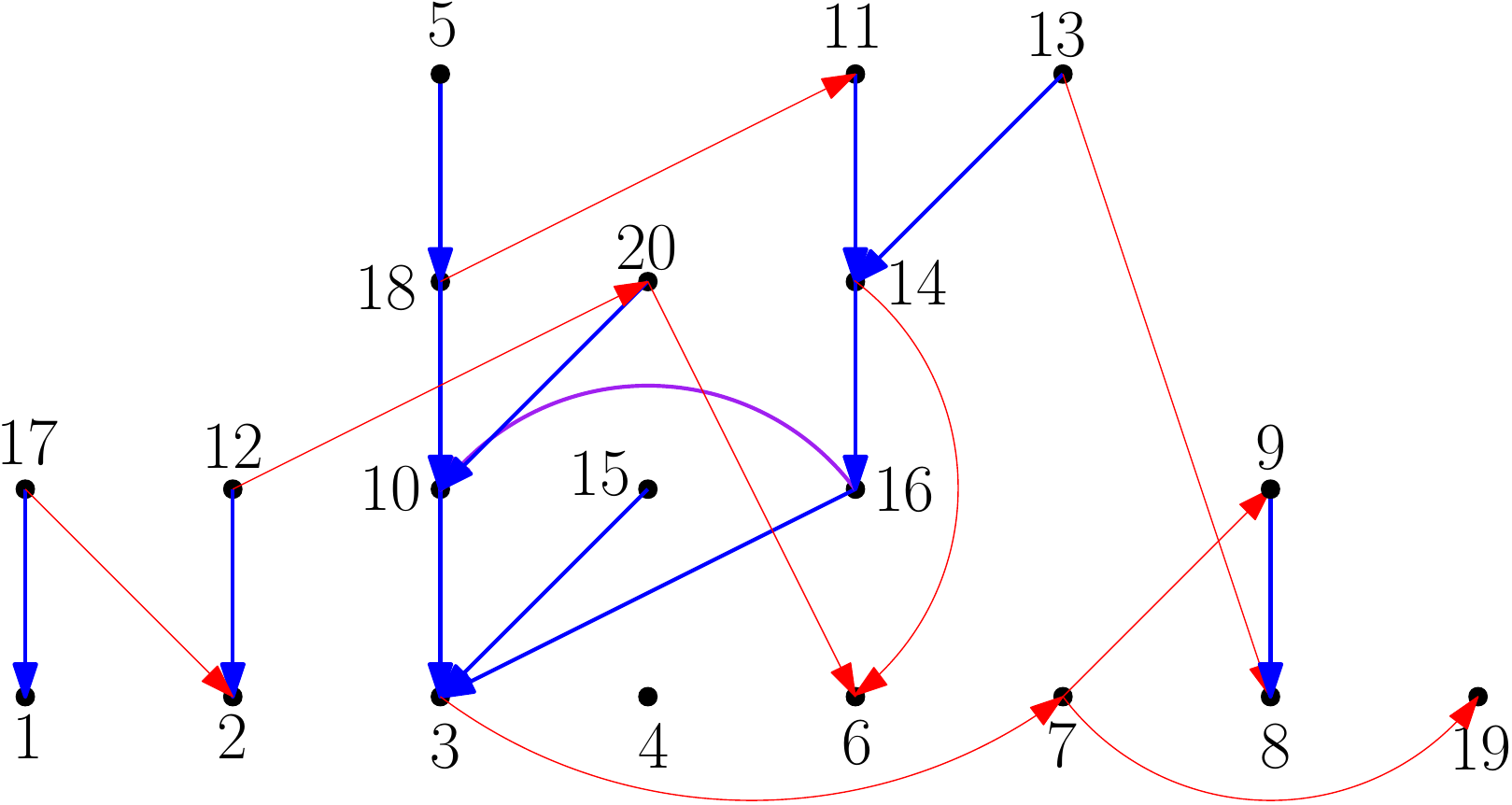}
\caption{A DFS decomposition of a random digraph $\vec G$ on $\set{20}$. The DFS order of $\vec G$ is $1\prec 17\prec 2 \prec 12\prec 3 \prec 10 \prec 18\prec 5 \prec 20\prec 15 \prec 14 \prec 11 \prec 13 \prec 16 \prec 4 \prec 6 \prec 7 \prec 8 \prec 9 \prec 19$. The down-edges $\downarrow(\vec G)$ are blue, the surplus-edges $\hookleftarrow(\vec G)$ are purple, and the right-edges $\leadsto(\vec G)$ are red. The longest saop has length $9$ and is $(5,18, 11, 14, 16, 10, 3, 7, 9, 8)$.
} 
\label{explore1} \label{DFSfig}
\end{figure}
\subsection{Analog problem for uniform rooted trees, and why it implies Proposition \ref{Main}} \label{2.3} 
Recall Algorithm \ref{algo2}. We write $\C(1)$ for the connected component of $1$ in $G(n,p)$.
Note that all saop in $\vec G$ ending at $1$ must start in $\C(1)$. So we may focus on $\C(1)$. Since with small probability $\C(1)$ contains a cycle, we may restrict to the case where $\C(1)$ is a tree. In this case, given $\# \C(1)$, it is well known that $\C(1)$ is a uniform rooted tree. 
For those reasons to show Proposition \ref{Main} it suffices to prove the following analog problem (see the end of the section for more details):

For $m\in \N$, $p\in [0,1]$, we construct a random directed graph $\vec \T_{m,p}$ as follows: Let $\T_m$ be a uniform tree with vertices  $\set m$. Root $\T_m$ at $1$. Let $\vec \T_m$ be $\T_m$ with the edges oriented toward $1$. Let $\leadsto(\vec \T_{m,p})$ be an independent copy of $\vec G(m,p)$. Let $\vec \T_{m,p}$ be the union of $\vec \T_m$ and $\leadsto(\vec \T_{m,p})$.

We write $\L(\vec \T_{m,p})$ for the length of the longest saop in $\vec \T_{m,p}$. 
We need the following result:
\begin{proposition} \label{Main2} There exists some constants $C>0$ such that for every $m\in \N$, $p>0$, $x>0$, 
\[ \proba(\L(\vec \T_{m,p})>x \sqrt{m})\leq C(1/x^2+p^{2/3}m).
\]
\end{proposition}
\begin{remark} The method of the proof can yield a much better bound, but I do not think it is needed.
\end{remark}
\begin{proof}[Proof of Proposition \ref{Main} given Proposition \ref{Main2}.] By an union bound we get:
\begin{equation} \proba(\L^{\uparrow}(1)>\delta n^{1/3})\leq \proba(\C(1)\text{ is not a tree})+\sum_{m=1}^\infty \proba \left (\L^{\uparrow}(1)>\delta n^{1/3}, \C(1) \text{ is a tree}, \#\C(1)=m\right )  . \label{7/5/22h} \end{equation}
By classical results on Erd\"os-R\'enyi random graphs, in the critical window, $\C(1)$  is not a tree with probability $O(n^{-1/3})$ (see e.g. Aldous \cite{Aldous_exc_ER}). 
 And by Proposition \ref{Main2} for some universal constant $C$, for every $m\geq 1$,
\begin{align*} \proba \left (\L^{\uparrow}(1)>\delta n^{1/3}, \C(1) \text{ is a tree}, \#\C(1)=m\right ) & \leq \proba( \C(1) \text{ is a tree}, \#\C(1)=m)\proba(\L(\vec \T_{m,p})>\delta n^{1/3})
\\ & \leq \proba(\#\C(1)=m) Cm(1/(\delta^2 n^{2/3})+p^{2/3}). \end{align*}
It follows by sum that, writing $S$ for the sum in \eqref{7/5/22h},
\[ S\leq C (1/(\delta^2 n^{2/3})+p^{2/3}) \E[\#\C(1)]. \]
Moreover, it is well known that in the critical window $ \E[\#\C(1)]=O(n^{1/3})$ (see e.g. Aldous \cite{Aldous_exc_ER}). 
 Thus $S=O(n^{-1/3})$, and Proposition \ref{Main} follows from \eqref{7/5/22h}.
\end{proof}
\section{Proof of Proposition \ref{Main2}} \label{Sec:Chaining} 
In this section $m,p$ are fixed, and we often omit in the notations the dependence in $m,p$.

To prove Proposition \ref{Main2} we use the chaining method based on leaf tightness. This idea is not new and comes from the pioneer papers of Aldous  \cite{Aldous1,Aldous3} on uniform rooted trees. It has been refined over the last recent years in e.g. \cite{Amini, CurienHaas, Seni,BRphd}, and I refer to the introduction of my thesis \cite{BRphd} for a detailed discussion on the subject.

The main idea is that we may approximate all paths from $\vec \T_{m,p}$ by looking only at the paths starting from a finite number of vertices. To go in the  details, we need to introduce a few notations. We consider for $i\in \set m$, $\L(i)$ the length of the longest saop starting from $i$ in $\vec \T_{m,p}$. Similarly for $i\in \set m,S\subset \set m$, let $\L(i,S)$ be the length of the longest saop starting from $i$ in $\vec \T_{m,p}\cap S$. Let $\height(S)-1$ be the length of the longest saop in $\vec \T_m\cap S$.  Also let $\downarrow(i)$ be the set of ancestors of $i$ in $\T_m$. Our chaining argument is based on the next result:

\begin{lemma} \label{Principle2}
For all $S\subset \set m$, $j\in \set m$, 
\[ \L(j) \leq \L\left (j, \set m \backslash \bigcup_{i\in S}\downarrow(i) \right )+1+\max_{i\in S} \L(i). \] 
\end{lemma}
\begin{proof}
Note that a saop starting from $j$ must meet $ \bigcup_{i\in S}\downarrow(i)$ in at most $\L (S\backslash \bigcup_{i\in S}\downarrow(i) )+1$ steps, and after it meet say $\downarrow(i)$, it coincides with a part of a path starting from $i$ and so must last at more $\L(i)$ steps. 
\end{proof}
Morally the previous lemma tells us that for $S',S\subset \set m$ to estimate $\max_{i\in S'} \L(i)$ given $\max_{i\in S} \L(i)$ it suffices to estimate the length of the paths in a much smaller subgraph of $\vec \T_{m,p}$.  Those paths then tends to be much smaller. This allows us to inductively estimate for well chosen sets $S_1\subset S_2\subset \dots \subset S_K=\set m$, the $\max_{i\in S} \L(i)$. To do so, we need an upper-bound on $\L(i,S)$ for $i\in \set m, S\subset \set m$. This is provided by the next lemma:
\begin{lemma} \label{Single} For every $\T_m$-measurable random variables $S\subset \set m, i\in \set m, k\in \N$,
\[ \proba(\L(i,S)>(k+1)\height(S)|\T_m)\leq (p \height(S)m)^k.\]
\end{lemma}
\begin{proof} Recall that a saop in $S$ which only use oriented edges from $\vec \T_m$ has length at most $\height(S)-1$. Thus it is enough to upper-bound the probability that there exists a path in $S$ starting from $i$ which uses at least $k$ edges from $\leadsto(\vec \T_{m,p})$. By an union bound this probability is bounded by
\[  \sum \proba(\forall i\in \llbracket 0, k-1\rrbracket, (j_a,i_{a+1})\in \leadsto(\vec T_{m,p})|\T_m),\] 
where the sum above is over all tuples $i_0,j_0,i_1,\dots i_k$ such that $(j_0,i_0)\neq (j_1,i_1)\neq \dots (j_k,i_k)$, and such that for every $a\in \llbracket 0,k-1\rrbracket$, there is a path from $i_{a}$ to $j_a$ in $\vec T_m \cap S$. Note that there is at most $\height(S)$ choices for $j_0$, then at most $m$ choices for $i_1$, then at most $\height(S)$ choices for $j_1$, and so on\dots And by definition of $ \leadsto(\vec T_{m,p})$, the above probability is $p^k$. The desired result follows.
\end{proof}
To finish the proof of Proposition \ref{Main2}, we need the next estimates. For $\delta>0$, let $N(\T_m,\delta)$ be the size of the minimal set $S\subset \set m$ such that $\height(\set m\backslash \bigcup_{i\in S} \downarrow (i))\leq x$. This covering number morally represents how many branches one need to delete in $\T_m$ to have a forest of height at most $x$.
\begin{lemma} \label{prelim} There exists a constant $C$ such that for every $m\in \N$, $\delta >0$,
\[ \proba \left ( N(\T_m, \delta \sqrt{m})\geq 1/\delta^{100}  \right )\leq C\delta^4,\]
and such that for every $m\in \N$, $x>0$,
\[ \proba \left ( \height(\set m)\geq x\sqrt{ m} \right )\leq C/x^4.\]
\end{lemma}
\begin{proof} The second part is well known. The first may be proved following readily Aldous \cite{Aldous1}. (Aldous only prove that some events hold with high probability but the same method still applies.) See also  \cite{Uniform} where similar bounds are proved in a wider context. We omit the details. 
\end{proof}
\begin{proof}[Proof of Proposition \ref{Main2}] Fix $m,p,x>0$. Without loss of generality up to chose a larger  $C$ in Proposition \ref{Main2}, we may assume $x,m$ large enough and $p^{2/3}m \geq 1/x^2\geq 1$. 
Let $l=(p^{2/3}m)^{-1/4}$. Note that $l\geq 1$. We first consider the following event:
\begin{equation}  \height(\set m)\leq l\sqrt{m}  \quad ; \quad \forall i\geq 1, \, N(\T_m,l \sqrt{m}/(2^i)\leq (2^il)^{100}. \tag{E} \label{proof:E}\end{equation}
By Lemma \ref{prelim}, we have for some universal constant $C$,
\begin{equation} \proba(\text{E})\geq 1-Cl^{-4}-C\sum_{i\geq 1} 2^{-4i} l^{-4}\geq 1-2Cl^{-4}=1-2Cp^{2/3}m. \label{8/5/14h}\end{equation}

Now we work given $\T_m$ under \eqref{proof:E}. For $S\subset \set m$, let $\downarrow (S):=\bigcup_{i\in S} \downarrow (i)$. By \eqref{proof:E}, and since $l\geq 1$, we may consider some increasing sets $\emptyset =S_0\subset S_1\subset S_2 \subset \dots$ such that for every $i\geq 1$, 
\[ \height(\set m\backslash \downarrow (S)) \leq l\sqrt{m}/2^i \quad \text{and} \quad \#(S_{i+1}\backslash S_i)\leq (2^il)^{100}.\]
 Also since $l\geq 1$, we may assume without loss of generality that for $k\geq \log(m)$, $S_k=\set m$. By Lemma \ref{Principle2}, for every $k\geq 0$,
\[ \max_{i\in S_{k+1}} \L(i)\leq 1+\max_{i\in S_{k+1}\backslash S_k} \L\left (i, \set m \backslash \downarrow(S_k) \right ) +\max_{i\in S_{k}} \L(i) .\]
Summing the last inequality over all $1\leq k \leq \log(m)$ we get
\begin{equation} \L(\vec \T_{m,p}) \leq  \log m+\sum_{k=0}^{\log m} \max_{i\in S_{k+1}\backslash S_k} \L\left (i, \set m \backslash \downarrow(S_k) \right ). \label{22/04/15h}\end{equation}

Moreover, by an union bound in Lemma \ref{Single}, for every $k\geq 0$, writing $S'_k:=\set m \backslash \downarrow(S_k)$,
\begin{align*} \proba(\exists i\in S_{k+1}\backslash S_k, \, \, \L(i,S'_k)> 102\height(S'_k) |\T_m) & \leq \#(S_{k+1}\backslash S_k)(p \height(S'_k) m)^{101}
\\ & \leq (2^kl)^{100} (pm^{3/2} l/2^k)^{101}. 
\\ & \leq 2^{-k} p^{2/3}m,
\end{align*}
using $l=(p^{2/3}m)^{-1/4}$ and $p^{2/3}m\leq 1$ for the last inequality.
By an union bound over all $k\geq 1$ we get that, given $\T_m$, with probability at least $1-2p^{2/3}m$, for every 
$k\geq 0, i\in S_{k+1}\backslash S_k$, we have $\L(i,S'_k)\leq 3 \height(S'_k)$. Hence, by \eqref{22/04/15h}, given $\T_m$ under \eqref{proof:E}, with probability at least $1-2p^{2/3}m$,
\[ \L(\vec \T_{m,p})   \leq  \log m+\sum_{k=0}^{\log m} \max_{i\in S_{k+1}\backslash S_k} \L\left (i, S'_k \right )
 \leq \log m+\sum_{k=0}^\infty \frac{102  \sqrt{m} l}{2^i}
 \leq \log m+102\sqrt{m} l \leq x\sqrt{m},\]
 where the last inequality holds when $p^{2/3} m\geq 1/x^2$ and $x,m$ are large enough. The desired result follows from \eqref{8/5/14h} and the above inequality.
\end{proof}
\bibliographystyle{plain}

\end{document}